\documentclass[a4paper,11pt]{amsart}
\usepackage[utf8]{inputenc}
\usepackage{amsmath,amssymb, comment}
\usepackage{graphicx,multirow,array,enumerate,wasysym}
\usepackage[left=3.6cm,right=3.6cm,bottom=3.6cm,top=3.6cm]{geometry}
\usepackage{amsthm}
\usepackage{natbib}
\usepackage{tikz}
\usepackage{microtype}
\usepackage{mathtools}
\usepackage{hyperref}

\bibpunct{[}{]}{,}{n}{}{;}

\newtheorem{theorem}{Theorem}[section]
\newtheorem{lemma}[theorem]{Lemma}
\newtheorem{proposition}[theorem]{Proposition}

\newtheorem{corollary}[theorem]{Corollary}
\theoremstyle{definition}

\newtheorem{example}[theorem]{Example}

\newcommand{\SZ}{\mathbb{Z}}                    
\newcommand{\SC}{\mathbb{C}}                    
\newcommand{\SP}{\mathbb{P}}                    %
\newcommand{\SA}{\mathbb{A}}                    %
\newcommand{\CZ}{\mathcal{Z}}                    %
\newcommand{\CC}{\mathcal{C}}                    %
\newcommand{\CF}{\mathcal{F}}                    

\newcommand{\frakg}{\mathfrak{g}}
\newcommand{\frakheis}{\mathfrak{heis}}
\newcommand{\ra}[1]{\kern-1.5ex\xrightarrow{\ \ #1\ \ }\phantom{}\kern-1.5ex}
\newcommand{\ras}[1]{\kern-1.5ex\xrightarrow{\ \ \smash{#1}\ \ }\phantom{}\kern-1.5ex}

\title{Young walls and equivariant Hilbert schemes of points in type $D$ }

\author{Ádám Gyenge}
\address{Mathematical Institute, University of Oxford, Andrew Wiles Building, Woodstock Road, OX2 6GG, Oxford, UK}
\email{Adam.Gyenge@maths.ox.ac.uk}

\begin{document}

\begin{abstract}
We give a combinatorial proof for a multivariable formula of the generating series of type D Young walls. Based on this we give a motivic refinement of a formula for the generating series of Euler characteristics of Hilbert schemes of points on the orbifold surface of type D. 
\end{abstract}

\maketitle


\section{Introduction}

In this paper we survey and refine some existing formulas expressing a connection between affine Lie algebras, Young diagram combinatorics and singularity theory as investigated in \cite{gyenge2015announcement,gyenge2015euler}. This connection is in the context of Hilbert schemes of points on orbifold surface singularities.

Let $G \subset SL(2,\mathbb{C})$ be a finite subgroup. The equivariant Hilbert scheme $\mathrm{Hilb}([\SC^2/G])$ is the moduli space of $G$-invariant finite colength subschemes of $\SC^2$, the invariant part of $\mathrm{Hilb}(\SC^2)$ under the lifted action of~$G$. 
This space decomposes as
\[ \mathrm{Hilb}([\SC^2/G])=\bigsqcup_{\rho \in {\mathop{\rm Rep}}(G)}\mathrm{Hilb}^{\rho}([\SC^2/G])\]
where
\[\mathrm{Hilb}^{\rho}([\SC^2/G])=  \{ I \in \mathrm{Hilb}(\SC^2)^G \colon H^0(\mathcal{O}_{\SC^2}/I) \simeq_G \rho \}\]
for any finite-dimensional representation $\rho\in {\mathop{\rm Rep}}(G)$ of $G$;
here $\mathrm{Hilb}(\SC^2)^G$ is the set of $G$-invariant ideals of $\SC[x,y]$, and $\simeq_G$ means $G$-equivariant isomorphism. Being components of fixed point sets of a finite group acting on smooth quasiprojective varieties, the orbifold Hilbert schemes themselves are smooth and quasiprojective \cite{cartan1957quotient}.

The topological Euler characteristics of the equivariant Hilbert scheme can be collected into a generating function. Let $\rho_0,\ldots,\rho_n\in\mathop{\rm Rep}(G)$ denote the (isomorphism classes of) irreducible representations of $G$, with $\rho_0$ the trivial representation.
The \textit{orbifold generating series} of the orbifold $[\SC^2/G]$ is 
\[Z_{[\SC^2/G]}(q_0,\ldots, q_n)= \sum_{m_0,\dots,m_n=0}^\infty \chi\left(\mathrm{Hilb}^{m_0 \rho_0 + \ldots +m_n \rho_n}([\SC^2/G]) \right)   q_0^{m_0}\cdots q_n^{m_n}\]
where $\rho_0,\ldots,\rho_n$ are the irreducible representations of $G_\Delta$, and  $q_0,\dots,q_n$ are formal variables.


It turns out that if $\SC^2/G$ is a simple (Kleinian, DuVal) singularity, then the series $Z_{[\SC^2/G]}(q_0,\ldots, q_n)$ is closely connected to character formulas of certain affine Lie algebras. In general, for an affine Lie algebra $\tilde\frakg$, denote by $\widetilde{\frakg\oplus\SC}$ the  Lie algebra that is 
the direct sum of it and an infinite Heisenberg algebra $\frakheis$, 
with their centers identified. Let $V_0$ be the basic representation of $\tilde\frakg$. 
Let $\CF$ be the standard
Fock space representation of $\frakheis$, having central charge 1. Then $V=V_0\otimes \CF$ is a representation of 
$\widetilde{\frakg\oplus\SC}$ that is called the {\em extended basic representation}.
A distinguished basis of this representation is was introduced by Kashiwara in the context of 
the associated quantum groups; this is known as the ``crystal basis''.

Suppose that the affine Lie algebra $\tilde\frakg$ is the extended loop algebra of a simple finite dimensional Lie algebra of simply laced type. 
The simply laced finite type Lie algebras are classified into three families: type $A_n$ for $n\geq 1$, type $D_n$ for $n \geq 4$  and type $E_n$ for $n=6,7,8$. 
To each type (or root system) $\Delta$ there also corresponds a finite subgroup of $SL(2, \mathbb{C})$; this will be denoted by $G_{\Delta}$.

It is known that the equivariant Hilbert schemes $\mathrm{Hilb}^{\rho}([\SC^2/\Gamma])$ for all finite dimensional representations $\rho$ of $G$ are Nakajima quiver varieties~\cite{nakajima2002geometric} associated to $\widetilde\Delta$, with dimension vector determined by $\rho$, and a specific stability condition (see \cite{fujii2005combinatorial, nagao2009quiver} for more details for type $A$).
The results of~\cite{nakajima2002geometric} on the relation between the cohomology of quiver varieties and affine Lie algebras, specialized to this case, imply that the direct sum of all cohomology groups $H^*(\mathrm{Hilb}^{\rho}([\SC^2/G]))$ is graded isomorphic to the extended basic representation~$V$ of the corresponding extended affine Lie algebra $\widetilde{\frakg\oplus\SC}$. 
This result combined with the Weyl-Kac character formula for the extended basic representation gives the following formula (see \cite[Appendix A]{gyenge2015euler}):
\begin{gather} Z_{[\SC^2/G_\Delta]}(q_0,\dots,q_n)= \\\left(\prod_{m=1}^{\infty}(1-q^m)^{-1}\right)^{n+1}\cdot\sum_{ \mathbf{m}=(m_1,\dots,m_n) \in \SZ^n } q_1^{m_1}\cdots q_n^{m_n}(q^{1/2})^{\mathbf{m}^\top \cdot C_\Delta \cdot \mathbf{m}}\label{eq:orbi_main_formula}
\end{gather}
where $q=\prod_{i=0}^n q_i^{d_i}$ with $d_i=\dim\rho_i$, 
and $C_\Delta$ is the finite type Cartan matrix corresponding to $\Delta$.

At least in types A and D an even stronger statement can be obtained. In these cases the elements of the crystal basis are in bijection with certain combinatorial objects called Young walls of type $\Delta$. The set of Young walls of type $\Delta$ will be denoted as $\mathcal{Z}_{\Delta}$; these are endowed with an $n+1$ dimensional multi-weight: $\mathbf{wt}(\lambda)=(\mathrm{wt}_0(\lambda),\dots, \mathrm{wt}_n(\lambda))$. The multi-variable generating series of objects in $\mathcal{Z}_{\Delta}$ is
\[Z_{\Delta}(q_0,\dots,q_n) = \sum_{\lambda\in \CZ_\Delta} \mathbf{q}^{\mathbf{wt}(\lambda)}\]
where we used the multi-index notation 
\[\mathbf{q}^{\mathbf{wt}(\lambda)}=\prod_{i=0}^nq_i^{\mathrm{wt}_i(\lambda)}.\] 
An important property of Young walls in types A and D is that there is a bijection
\begin{equation} \CZ_\Delta  \longleftrightarrow  {\mathcal P}^{n+1}  \times  \SZ^n\end{equation}
where ${\mathcal P}$ is the set of ordinary partitions and $n$ is the rank of the root system. In type A this goes back to \cite{james1981representation}; for the type D case see Proposition \ref{prop:dncoredecomp} below. This serves as the starting point of the following enhancement of \eqref{eq:orbi_main_formula}.
\begin{theorem} Let $\Delta$ be of type A or D.
\label{thm:main}
\begin{enumerate}
\item 
\begin{gather*} Z_{\Delta}(q_0,\dots,q_n)= \\ \left(\prod_{m=1}^{\infty}(1-q^m)^{-1}\right)^{n+1}\cdot\sum_{ \mathbf{m}=(m_1,\dots,m_n) \in \SZ^n } q_1^{m_1}\cdots q_n^{m_n}(q^{1/2})^{\mathbf{m}^\top \cdot C_\Delta \cdot \mathbf{m}}.\end{gather*}
\item There exist a locally closed decomposition
of $\mathrm{Hilb}([\SC^2/G_\Delta])$ into strata indexed by the elements of $\CZ_\Delta$.
Each stratum is isomorphic to an affine space.
\item In particular,
\[Z_{[\SC^2/G_\Delta]}(q_0,\dots,q_n)=Z_{\Delta}(q_0,\dots,q_n).\]
\end{enumerate}
\end{theorem}
The type A case of this statement was proved in \cite{fujii2005combinatorial}.
Part (2) for type D was proved in \cite{gyenge2015euler}. The statement of part (1) in type D was also stated in \cite{gyenge2015euler}, and a proof was sketched in \cite{gyenge2016phdthesis}; we flesh this out in Sections \ref{sect:dnabacus}--\ref{seq:enumyw} below.

One can also consider a motivic enhancement of the series considered above. Let $K_0(\mathrm{Var})$ be Grothendieck ring of quasi-projective varieties over the complex numbers. The \emph{motivic Hilbert zeta function} of the orbifold $[\SC^2/G_\Delta]$ is 
\[ \mathcal{Z}_{[\SC^2/G_\Delta]}(q_0,\ldots, q_n)= \sum_{m_0,\dots,m_n=0}^\infty [\mathrm{Hilb}^{m_0 \rho_0 + \ldots +m_n \rho_n}([\SC^2/G_\Delta]) ]   q_0^{m_0}\cdots q_n^{m_n}. \]
Here $[X]$ denotes the class of $X$ in $K_0(\mathrm{Var})$, and it is not to be confused with orbifold quotients. The series $\mathcal{Z}_{[\SC^2/G_\Delta]}(q_0,\ldots, q_n)$ is an element in $K_0(Var)[[q_0,\dots,q_n]]$.

The combination of \cite[Corollary 1.11]{bryan2019g} with Theorem \ref{thm:main} gives an explicit representation for the motivic Hilbert zeta function.
\begin{corollary} Let $\Delta$ be of type A or D.  
\label{thm:motive}
\begin{gather*}\mathcal{Z}_{[\SC^2/G_\Delta]}(q_0,\dots,q_n)= \\ \left(\prod_{m=1}^{\infty}(1-\mathbb{L}^{m+1}q^m)^{-1}(1-\mathbb{L}^{m}q^m)^{-n}\right)\cdot\sum_{ \mathbf{m}=(m_1,\dots,m_n) \in \SZ^n } q_1^{m_1}\cdots q_n^{m_n}(q^{1/2})^{\mathbf{m}^\top \cdot C_\Delta \cdot \mathbf{m}}
\end{gather*}
where $\mathbb{L}  = [\mathbb{A}^{1}]\in
K_{0}(\mathrm{Var})$.
\end{corollary}
Once again, in type A this statement was proved in \cite{fujii2005combinatorial}. The above series has further specializations giving formulas for the  Hodge polynomials and
Poincar\'e polynomials of the equivariant Hilbert schemes.

Let $Y \subset \SC^2$ be a closed subvariety invariant under the action of $G_\Delta$. One can consider the moduli space $\mathrm{Hilb}([\SC^2/G_\Delta],Y)\subset \mathrm{Hilb}([\SC^2/G_\Delta])$ of points supported on $Y$. The corresponding motivic generating series is 
\[ \mathcal{Z}_{([\SC^2/G_\Delta],Y)}(q_0,\ldots, q_n)= \sum_{m_0,\dots,m_n=0}^\infty [\mathrm{Hilb}^{m_0 \rho_0 + \ldots +m_n \rho_n}([\SC^2/G_\Delta],Y) ]   q_0^{m_0}\cdots q_n^{m_n}. \]
The techniques of \cite{gusein2004power} imply that
\begin{equation*}\label{eq:hilbmot} \mathcal{Z}_{[\SC^2/G_\Delta]}(q_0,\dots,q_n)=\mathcal{Z}_{([\SC^2/G_\Delta],Y)}(q_0,\dots,q_n)\cdot \mathcal{Z}_{[(\SC^2\setminus Y)/G_\Delta]}(q_0,\dots,q_n). \end{equation*}
This allows one to obtain further formulas from Corollary \ref{thm:motive}. For example,
\begin{gather*}\mathcal{Z}_{([\SC^2/G_\Delta],0)}(q_0,\dots,q_n)= 
\frac{\mathcal{Z}_{[\SC^2/G_\Delta]}(q_0,\dots,q_n)}{\mathcal{Z}_{[(\SC^2\setminus 0) /G_\Delta]}(q_0,\dots,q_n)}\\
= 
\frac{\mathcal{Z}_{[\SC^2/G_\Delta]}(q_0,\dots,q_n)}{\mathcal{Z}_{(\SC^2\setminus 0) /G_\Delta}(q)}
=\frac{\mathcal{Z}_{[\SC^2/G_\Delta]}(q_0,\dots,q_n)}{\prod_{m=1}^{\infty}(1-\mathbb{L}^{m+1}q^m)^{-1}(1-\mathbb{L}^{m-1}q^m)}=
\\ \left(\prod_{m=1}^{\infty}(1-\mathbb{L}^{m-1}q^m)^{-1}(1-\mathbb{L}^{m}q^m)^{-n}\right)\cdot\sum_{ \mathbf{m}=(m_1,\dots,m_n) \in \SZ^n } q_1^{m_1}\cdots q_n^{m_n}(q^{1/2})^{\mathbf{m}^\top \cdot C_\Delta \cdot \mathbf{m}},
\end{gather*}
where at the second equality we have used that $G_\Delta$ acts freely away from the origin, and at the third equality we have used the main result of \cite{gottsche1990betti} and that $[(\SC^2 \setminus {0}) /G_\Delta]=[\mathbb{L}^2]-[pt]$ in $K_0(\mathrm{Var})$.

Suppose that $\Delta$ is of type D. Let $E \subset \SC^2$ be the divisor defined by the ideal $(xy)$. This is invariant under the action of $G_\Delta$. Then
\begin{gather*}\mathcal{Z}_{([\SC^2/G_\Delta],E)}(q_0,\dots,q_n)= 
\mathcal{Z}_{([\SC^2/G_\Delta],0)}(q_0,\dots,q_n)\cdot \mathcal{Z}_{([\SC^2/G_\Delta],E \setminus {0})}(q_0,\dots,q_n)\\
=\mathcal{Z}_{([\SC^2/G_\Delta],0)}(q_0,\dots,q_n)\cdot \mathcal{Z}_{(\SC^2/G_\Delta,(E \setminus {0})/G_\Delta)}(q)\\
=\mathcal{Z}_{([\SC^2/G_\Delta],0)}(q_0,\dots,q_n)\cdot \prod_{m=1}^{\infty}(1-\mathbb{L}^{m}q^m)^{-1}(1-\mathbb{L}^{m-1}q^m)
\\ =\left(\prod_{m=1}^{\infty}(1-\mathbb{L}^{m}q^m)^{-n-1}\right)\cdot\sum_{ \mathbf{m}=(m_1,\dots,m_n) \in \SZ^n } q_1^{m_1}\cdots q_n^{m_n}(q^{1/2})^{\mathbf{m}^\top \cdot C_\Delta \cdot \mathbf{m}},
\end{gather*}
where again at the second equality we have used that $G_\Delta$ acts freely away from the origin, and at the third equality we have used $[(E \setminus {0}) /G_\Delta]=[\mathbb{L}]-[pt]$ in $K_0(\mathrm{Var})$.
\begin{corollary} 
\begin{enumerate}
	\item 
There exist a locally closed decomposition
of $\mathrm{Hilb}([\SC^2/G_\Delta],E)$ into strata indexed by the elements of $\CZ_\Delta$.
Each stratum is isomorphic to an affine space. 
\item The class in  $K_0(\mathrm{Var})$ of the stratum $\mathrm{Hilb}([\SC^2/G_\Delta],E)_Y$ corresponding to a Young wall $Y=(\lambda_1,\dots,\lambda_{n+1},\mathbf{m}) \in \CZ_\Delta \cong {\mathcal P}^{n+1}  \times  \SZ^n$ is 
\[[\mathrm{Hilb}([\SC^2/G_\Delta],E)_Y]=[\mathbb{L}]^{\sum_{i=1}^{n+1} |\lambda_i|},\]
where $|\lambda_i|=\sum_j\lambda_i^j$.
\end{enumerate}	
\end{corollary}
\begin{proof}
The proof of Part (1) is very similar to that of \cite[Theorem 4.1]{gyenge2015euler}. The divisor $E$ is preserved by the diagonal torus action on $\SC^2$ used in \cite{gyenge2015euler} for the stratification of $\mathrm{Hilb}([\SC^2/G_\Delta])$. It follows that the torus action on $\mathrm{Hilb}([\SC^2/G_\Delta],E)$ has the same fixed points as the torus action on $\mathrm{Hilb}([\SC^2/G_\Delta])$. By \cite[Theorem 4.3]{gyenge2015euler},
\[\mathrm{Hilb}([\SC^2/G_\Delta])^{\SC^{\ast}}=\bigsqcup_{Y \in \CZ_\Delta}Z_Y\]
where each $Z_Y$ is an affine space. Let $\mathrm{Hilb}([\SC^2/G_\Delta],E)_Y \subset \mathrm{Hilb}([\SC^2/G_\Delta],E)$ denote the locus of ideals which flow to $Z_Y$ under the torus action.
Since $(E \setminus {0}) /G_\Delta \cong \SC^{\ast}$, the Zariski locally trivial fibration
$\mathrm{Hilb}([\SC^2/G_\Delta])_Y \to Z_Y$ explored in \cite[Theorem 4.1]{gyenge2015euler} restricts to a Zariski locally trivial fibration $\mathrm{Hilb}([\SC^2/G_\Delta],E)_Y \to Z_Y$ with affine space fibers, and a compatible torus action on the fibers. By \cite[Sect.3, Remarks]{bialynicki1973some} this fibration is an algebraic vector bundle over $Z_Y$, and hence trivial (Serre–Quillen–Suslin).

Part (2) follows from Part (1) and the formula for $\mathcal{Z}_{([\SC^2/G_\Delta],E)}(q_0,\dots,q_n)$ above.
\end{proof}

The aim of the current paper is twofold. First, we give an exposition about the combinatorics of the Young walls in type D. Second, we give a complete proof of Theorem \ref{thm:main} (1) (and hence (3)) in the type D case.

The structure of the rest of the paper is as follows.  
In Section 2 we review the combinatorics of the Young walls in type D. In Section 3 we introduce an associated combinatorial tool called the abacus. Using this we will calculate the generating series $Z_{\Delta}(q_0,\dots,q_n)$ of Young walls of type D and prove Theorem \ref{thm:main} (1) in Section 4.



\subsection*{Acknowledgement} The author is thankful to Jim Bryan, Andr\'as N\'emethi and Bal\'azs Szendr\H{o}i for fruitful conversations about the topic.

\section{Young walls of type \texorpdfstring{$D_n$}{Dn}}
\label{Dnidealsect}

It is known that when $\Delta=A_n$, $n\geq 1$, the Young wall $\mathcal{Z}_{\Delta}=\mathcal{P}$, the set of all Young diagrams/partitions equipped with the diagonal coloring (see \cite{gyenge2017enumeration}). We describe here the type $D$ analogue of the set of diagonally colored partitions used in type $A$, following \cite{kang2004crystal,kwon2006affine}.

First we define the {\em Young wall pattern of type 
 $D_n$}. 
This is the following infinite
pattern, consisting of two types of blocks: half-blocks carrying possible labels
$j\in\{0,1,n-1, n\}$, and full blocks carrying possible labels $1<j<n-1$:

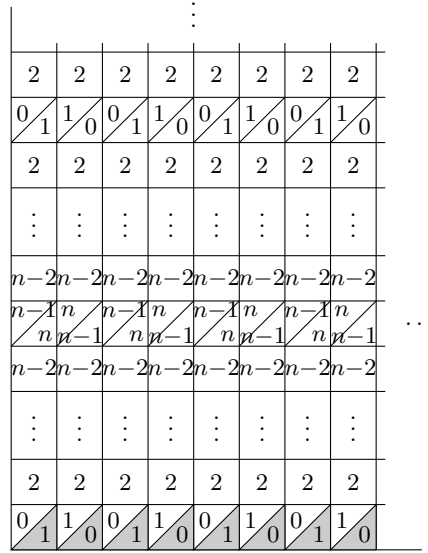
\begin{center}
\begin{tikzpicture}[scale=0.6, font=\footnotesize, fill=black!20]
  \foreach \x in {1,2,3,4,5,6,7,8}
    {
      \draw (\x, 0) -- (\x,11+0.2);
    }
     \draw (0, 0) -- (0,12);
   \foreach \y in {1,2,3.5,4.5,5.5,6.5,8,9,10,11}
    {
         \draw (0,\y) -- (8.2,\y);
    }
    \draw (0,0) -- (9,0);
    \foreach \x in {0,1,2,3,4,5,6,7}
    {
    	\draw (\x,4.5) -- (\x+1,5.5);
    	\draw (\x,9) -- (\x+1,10);
    	\draw (\x+0.5,1.5) node {2};
    	\draw (\x+0.5,4) node {$n$$-$$2$};
    	\draw (\x+0.5,6) node {$n$$-$$2$};
    	\draw (\x+0.5,8.5) node {2};
    	\draw (\x+0.5,10.5) node {2};
    	\draw(\x+0.5,2.85) node {\vdots};
    	\draw(\x+0.5,7.35) node {\vdots};
    	\filldraw (\x,0) -- (\x+1,1) -- (\x+1,0) -- cycle ;
    }
     \foreach \x in {0,2,4,6}
        {
        	\draw (\x+0.25,0.65) node {0};
        	\draw (\x+0.75,0.35) node {1};
        	\draw (\x+0.49,5.28) node  {$n$$-$$1$};
        	\draw (\x+0.75,4.72) node {$n$};
        	\draw (\x+0.25,9.65) node {0};
        	\draw (\x+0.75,9.35) node {1};
        }
       \foreach \x in {1,3,5,7}
             {
             	\draw (\x+0.25,0.65) node {1};
             	\draw (\x+0.75,0.35) node {0};
             	\draw (\x+0.25,5.28) node {$n$};
             	\draw (\x+0.52,4.72) node {$n$$-$$1$};
             	\draw (\x+0.25,9.65) node {1};
             	\draw (\x+0.75,9.35) node {0};
             }
             
       \draw (9,5) node {\dots};
       \draw (4,12) node {\vdots};
\end{tikzpicture}
\end{center}

A {\em Young wall\footnote{In \cite{kang2004crystal,kwon2006affine}, these arrangements are called {\em proper Young walls}. Since we will not meet any other Young wall, we will drop the adjective {\em proper} for brevity.} of type $D_n$} is a 
subset~$Y$ of the infinite Young wall of type $D_n$, satisfying the following rules. 
\begin{enumerate}
\item[(YW1)] $Y$ contains all grey half-blocks, and a finite number of the white blocks and half-blocks. 
\item[(YW2)] $Y$ consists of continuous columns of blocks, with no block placed on top of a missing block or half-block. 
\item[(YW3)] Except for the leftmost column, there are no free
  positions to the left of any block or half-block. Here the rows of
  half-blocks are thought of as two parallel rows; only half-blocks of the same orientation have to be present.
\item[(YW4)] A full column is a column with a full block or both half-blocks present at its top; then no two full columns 
have the same height\footnote{This is the properness condition of \cite{kang2004crystal}.}.
\end{enumerate}

Let $\CZ_\Delta$ denote the set of all Young walls of type $D_n$. For any $Y\in \CZ_\Delta$ and label $j\in \{0, \ldots, n\}$ let $wt_j(Y)$ be the number of white half-blocks, respectively blocks, of label $j$. These are collected into the multi-weight vector $\mathbf{wt}(Y)=(wt_0(Y), \ldots, wt_n(Y))$. The total weight of $Y$ is the sum
\[|Y|=\sum_{j=0}^{n} wt_j(Y),
\]
and for the formal variables $q_0,\dots,q_n$,
\[\mathbf{q}^{\mathbf{wt}(Y)}=\prod_{j=0}^{n}q_j^{wt_j(Y)}.\]
\begin{example} 
\label{ex:ywall}	
The following is an example of a Young wall for $\Delta=D_4$:
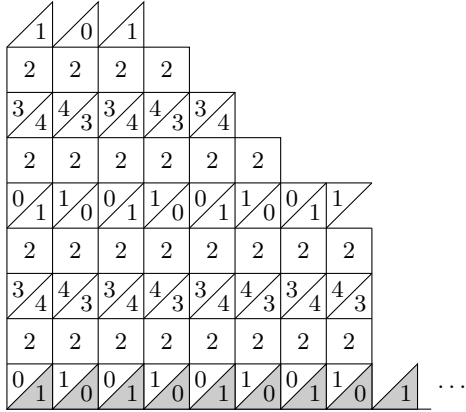
\begin{center}
\begin{tikzpicture}[scale=0.6, font=\footnotesize, fill=black!20]

\draw (0, 0) -- (0,8) -- (1,9)--(1,8);
\draw (1, 0) -- (1,8) -- (2,9)--(2,8);
\draw (2, 0) -- (2,8) -- (3,9)--(3,8);
\draw (3, 0) -- (3,8);
\draw (4, 0) -- (4,8) -- (0,8);
\draw (5, 0) -- (5,7) -- (0,7);
\draw (6, 0) -- (6,6) -- (0,6);
\draw (7, 0) -- (7,5) -- (0,5);
\draw (7, 5) -- (8,5);
\draw (8, 0) -- (8,4);
\draw (0, 4) -- (8,4);
\draw (0, 3) -- (8,3);
\draw (0, 2) -- (8,2);
\draw (0, 1) -- (8,1);
\draw (9, 0) -- (9.3,0);

\draw (0,0) -- (9,0);
\foreach \x in {0,1,2,3,4,5,6,7}
{
	\draw (\x,2) -- (\x+1,3);
	\draw (\x,4) -- (\x+1,5);
	\draw (\x+0.5,1.5) node {2};
	\draw (\x+0.5,3.5) node {2};
	\filldraw (\x,0) -- (\x+1,1) -- (\x+1,0) -- cycle ;
}
\filldraw (8,0) -- (9,1) -- (9,0) -- cycle ;
\draw (8.75,0.35) node {1};
\foreach \x in {0,2,4,6}
{
	\draw (\x+0.25,0.65) node {0};
	\draw (\x+0.75,0.35) node {1};
	\draw (\x+0.25,2.65) node {3};
\draw (\x+0.75,2.35) node {4};
	\draw (\x+0.25,4.65) node {0};
	\draw (\x+0.75,4.35) node {1};
}
\foreach \x in {1,3,5}
{
	\draw (\x+0.25,0.65) node {1};
	\draw (\x+0.75,0.35) node {0};
	\draw (\x+0.25,2.65) node {4};
\draw (\x+0.75,2.35) node {3};
	\draw (\x+0.25,4.65) node {1};
	\draw (\x+0.75,4.35) node {0};
}
	\draw (7.25,0.65) node {1};
\draw (7.75,0.35) node {0};
\draw (7.25,2.65) node {4};
\draw (7.75,2.35) node {3};
\draw (7.25,4.65) node {1};

\foreach \x in {0,2}
{
	\draw (\x+0.25,6.65) node {3};
	\draw (\x+0.75,6.35) node {4};
	\draw (\x+0.75,8.35) node {1};
	\draw (\x,6) -- (\x+1,7);
}
\draw (1.25,6.65) node {4};
\draw (1.75,6.35) node {3};
\draw (1.75,8.35) node {0};
\draw (3.25,6.65) node {4};
\draw (3.75,6.35) node {3};
\draw (1,6) -- (2,7);
\draw (3,6) -- (4,7);
\draw (4.25,6.65) node {3};
\draw (4.75,6.35) node {4};
\draw (4,6) -- (5,7);
\foreach \x in {0,1,2,3}
{
	\draw (\x+0.5,5.5) node {2};
	\draw (\x+0.5,7.5) node {2};
}
\draw (4.5,5.5) node {2};
\draw (5.5,5.5) node {2};
\draw (9.8,0.5) node {\dots};
\end{tikzpicture}
\end{center}
\end{example}

\section{Abacus combinatorics}
\label{sect:dnabacus}


Recalling the Young wall rules (YW1)-(YW4), it is clear that every $Y\in \CZ_\Delta$ can be decomposed as
$Y=Y_1 \sqcup Y_2$, where $Y_1\in \CZ_\Delta$ has full columns only, and $Y_2\in \CZ_\Delta$ has all its columns
ending in a half-block. These conditions define two subsets $\CZ^f_\Delta, \CZ^h_\Delta\subset \CZ_\Delta$.
Because of the Young wall rules, the pair $(Y_1, Y_2)$ uniquely reconstructs $Y$, so we get a bijection
\begin{equation}\label{decomp_D_YW}\CZ_\Delta \longleftrightarrow \CZ^f_\Delta\times \CZ^h_\Delta.\end{equation}

Given a Young wall $Y \in \mathcal{Z}_\Delta$ of type $D_n$, let $\lambda_k$ denote the number of 
blocks (full or half blocks both contributing 1) in the $k$-th vertical 
column. By the rules of Young walls, the resulting positive integers
$\{\lambda_1,\dots,\lambda_r\}$ form a partition $\lambda=\lambda(Y)$ of weight equal to the total weight $|Y|$, 
with the additional property that 
its parts $\lambda_k$ are distinct except when $\lambda_k \equiv 0\;\mathrm{mod}\; (n-1)$. 
Corresponding to the decomposition~\eqref{decomp_D_YW}, we get a decomposition 
$\lambda(Y) = \mu(Y)\sqcup \nu(Y)$. In $\mu(Y)$, no part is congruent to $0$ modulo $(n-1)$, and there 
are no repetitions; all parts in $\nu(Y)$ are congruent to $0$ modulo $(n-1)$ and repetitions are allowed. 
Note that the pair $(\mu(Y), \nu(Y))$ does almost, but not quite, encode $Y$, because of the ambiguity 
in the labels of half-blocks on top of non-complete columns.

We now introduce another combinatorial object, \emph{the abacus of type 
$D_n$} (\cite{kang2004crystal, kwon2006affine}). 
This is the arrangement of positive integers, called positions, in the following pattern:

\vspace{.1in}
\begin{center}
\begin{tabular}{c c c c c c c c}
1 & \dots & $n-2$ & $n-1$ & $n$ & \dots & $2n-3$ & $2n-2$  \\
$2n-1$ & \dots & $3n-4$ & $3n-3$ & $3n-2$ & \dots & $4n-5$ & $4n-4$\\
\vdots & &\vdots & \vdots & \vdots& &\vdots & \vdots
\end{tabular}
\end{center}
\vspace{.1in}

For any integer $1 \leq k \leq 2n-2$, the set of positions in the $k$-th column of the abacus is the \emph{$k$-th 
ruler}, denoted $R_k$. Several \emph{beads} are placed on these rulers. For $k \not\equiv 0\; \mathrm{mod}\; (n-1)$, 
the rulers $R_k$ can only contain normal (uncolored) beads,
with each position occupied 
by at most one bead. On the rulers $R_{n-1}$ and $R_{2n-2}$, the beads are colored white and black.
An arbitrary number of white or black beads can be put on each such position, 
but each position can only contain beads of the same color.

Given a type $D_n$ Young wall $Y \in \mathcal{Z}_\Delta$, let $\lambda=\mu\sqcup \nu$ be the corresponding 
partition with its decomposition. For each nonzero part $\nu_k$ of $\nu$, set 
\[ n_k=\#\{1 \leq j \leq l(\mu)\;|\;\mu_j < \nu_k \}\]
to be the number of full columns shorter than a given non-full column. 
The abacus configuration of the Young wall $Y$ is defined to be the set of beads placed at positions
$\lambda_1,\dots,\lambda_r$. The beads at positions $\lambda_k=\mu_j$ are uncolored; the
color of the bead at position $\lambda_k=\nu_l$ corresponding to a column $C$ of $Y$ is
\[
\begin{cases}
\textrm{white,} & \textrm{if the block at the top of } C \textrm{ is } 
\begin{tikzpicture}
 \draw (0, 0) -- (0.3,0.3);
 \draw (0, 0) -- (0.3,0);
 \draw (0.3, 0) -- (0.3,0.3);
\end{tikzpicture}
\textrm{ and } n_l \textrm{ is even;} \\ 
\textrm{white,} & \textrm{if the block at the top of } C \textrm{ is } 
\begin{tikzpicture}
 \draw (0, 0) -- (0.3,0.3);
 \draw (0, 0) -- (0,0.3);
 \draw (0, 0.3) -- (0.3,0.3);
\end{tikzpicture}
\textrm{ and } n_l \textrm{ is odd;}\\
\textrm{black,} & \textrm{if the block at the top of } C \textrm{ is } 
\begin{tikzpicture}
 \draw (0, 0) -- (0.3,0.3);
 \draw (0, 0) -- (0,0.3);
 \draw (0, 0.3) -- (0.3,0.3);
\end{tikzpicture}
\textrm{ and } n_l \textrm{ is even;} \\ 
\textrm{black,} & \textrm{if the block at the top of } C \textrm{ is } 
\begin{tikzpicture}
 \draw (0, 0) -- (0.3,0.3);
 \draw (0, 0) -- (0.3,0);
 \draw (0.3, 0) -- (0.3,0.3);
\end{tikzpicture}
\textrm{ and } n_l \textrm{ is odd.}\\
\end{cases}
\]
One can check that the abacus rules are satisfied, that all abacus configurations satisfying the above rules, 
with finitely many uncolored, black and white beads, can arise, and that the Young wall $Y$ is uniquely determined 
by its abacus configuration.

\begin{example} \label{ex:abacusY}
The abacus configuration associated with the Young wall  of Example \ref{ex:ywall} is
\begin{center}
\begin{tikzpicture}

\draw (0,5) node {$R_1$};
\draw (1,5) node {$R_{2}$};
\draw (2,5) node {$R_{3}$};
\draw (3,5) node {$R_{4}$};
\draw (4,5) node {$R_{5}$};
\draw (5,5) node {$R_{6}$};

\draw (-0.5,4.8) -- (5.5,4.8);
\draw (2.5,5.3) -- (2.5,2.5);
\draw (1.5,5.3) -- (1.5,2.5);
\draw (4.5,5.3) -- (4.5,2.5);
\draw (5.5,5.3) -- (5.5,2.5);

\node at ( 0,4.5)  {1};
\node at ( 1,4.5)  {2};
\node at ( 2,4.5)  {3};
\node at ( 3,4.5)  {4};
\node at ( 4,4.5)  {5};
\node at ( 5,4.5) [circle,draw,inner sep=0pt,minimum size=12pt,fill=gray!50] {6};
\node at ( 0,4) [circle,draw,inner sep=0pt,minimum size=12pt] {7};
\node at ( 1,4) [circle,draw,inner sep=0pt,minimum size=12pt] {8};
\node at ( 2,4) {9};
\node at ( 3,4) [circle,draw,inner sep=0pt,minimum size=12pt] {10};
\node at ( 4,4) [circle,draw,inner sep=0pt,minimum size=12pt] {11};
\node at ( 5,4) [circle,draw,inner sep=0pt,minimum size=12pt] {12};
\node at ( 5,4)  {$\phantom{1222}^3$};
\node at ( 0,3.5)  {13};
\node at ( 1,3.5)  {14};
\node at ( 2,3.5)  {15};
\node at ( 3,3.5)  {16};
\node at ( 4,3.5)  {17};
\node at ( 5,3.5)  {18};
\node at ( 0.5,3) {\vdots};
\node at ( 3.5,3) {\vdots};
\end{tikzpicture}
\end{center}
The superscript at 12 indicates that there are 3 white beads at that position.
\end{example}



We now introduce certain distinguished Young walls of type $D_n$, and 
a method to obtain them with moving the beads on the abacus. On the Young wall side, define a {\em bar} to be a connected set of blocks
and half-blocks, with each half-block occurring once and each block occurring twice. A Young wall 
$Y\in\mathcal{Z}_\Delta$ will be called a {\em core} Young wall, if no bar can be removed from it without 
violating the Young wall rules. For an example of bar removal, see \cite[Example 5.1(2)]{kang2004crystal}. Let ${\mathcal C}_\Delta\subset\mathcal{Z}_\Delta$ denote the set of all 
core Young walls. 

Based on the calculations of \cite{kang2004crystal, kwon2006affine} the following result was obtained in \cite[Proposition 7.2]{gyenge2015euler}. For completeness we include also its proof.
\begin{proposition}
\label{prop:dncoredecomp}
Given a Young wall 
$Y\in \mathcal{Z}_\Delta$, any complete sequence of bar removals through Young walls results in the same
core $\mathrm{core}(Y)\in {\mathcal C}_\Delta$, defining a map of sets
\[ \mathrm{core}\colon \mathcal{Z}_\Delta\to {\mathcal C}_\Delta.\]
The process can be described on the abacus, respects the decomposition~\eqref{decomp_D_YW}, and results
in a bijection
\begin{equation} \CZ_\Delta  \longleftrightarrow  {\mathcal P}^{n+1}  \times  \CC_\Delta\label{Dpart_biject}\end{equation}
where ${\mathcal P}$ is the set of ordinary partitions.
Finally, there is also a bijection
\begin{equation}\label{Dcore_biject} \CC_\Delta\longleftrightarrow \SZ^n.\end{equation}
\end{proposition} 
\begin{proof} Decompose $Y$ into a pair of Young walls $(Y_1, Y_2)$ as above. Let us first consider $Y_1$. 
On the corresponding rulers $R_k$, $k \not\equiv 0\; \textrm{ mod}\; (n-1)$,
the following steps correspond to bar removals \cite[Lemma 5.2]{kang2004crystal}.
\begin{enumerate}
 \item[(B1)] If $b$ is a bead at position $s>2n-2$, and there is no bead at position $(s-2n+2)$, then move $b$ 
one position up and switch the color of the beads at 
positions $k$ with $k \equiv 0\; \textrm{ mod}\; (n-1)$, $s-2n+2 < k < s$.
 \item[(B2)] If $b$ and $b'$ are beads at position $s$ and $2n-2-s$ ($1 \leq s \leq n-2$) respectively, then 
remove $b$ and $b'$ and switch the color of the beads 
at positions $k \equiv 0\; \textrm{ mod}\; (n-1), s< k < 2n-2-s$.
\end{enumerate}
Performing these steps as long as possible results in a configuration of beads on the rulers $R_k$ with 
$k \not\equiv 0\; \textrm{ mod}\; (n-1)$ with no gaps from above, and for $1 \leq s \leq n-2$, beads on only one
of $R_s$, $R_{2n-2-s}$. This final configuration can be uniquely described by an ordered set of 
integers $\{z_1, \ldots, z_{n-2}\}$, $z_s$ being the number of beads on $R_s$ minus the number of beads 
on $R_{2n-2-s}$ \cite[Remark 3.10(2)]{kwon2006affine}. In the correspondence \eqref{Dcore_biject} this gives $\SZ^{n-2}$. It turns out that the reduction steps in this part of the algorithm can be encoded 
by an $(n-2)$-tuple of ordinary partitions, with the summed weight of these 
partitions equal to the number of bars removed \cite[Theorem 5.11(2)]{kang2004crystal}. 

Let us turn to $Y_2$, represented on the rulers $R_k$, $k \equiv 0\; \textrm{ mod}\; (n-1)$. 
On these rulers the following steps correspond to bar removals \cite[Sections 3.2 and 3.3]{kwon2006affine}.
\begin{enumerate} 
\item[(B3)] Let $b$ be a bead at position $s\geq 2n-2$. If there is no bead at position $(s-n+1)$, and the beads at position $(s-2n+2)$ are of the same color as $b$, then shift $b$ up to position $(s-2n+2)$.
 \item[(B4)] If $b$ and $b'$ are beads at position $s\geq n-1$, then move them up to position $(s-n+1)$. If $s-n+1>0$ and this position already contains beads, then $b$ and $b'$ take that same color.
\end{enumerate}
During these steps, there is a boundary condition: there is an imaginary position $0$ in the rightmost column, 
which  is considered to contain invisible white beads; placing a bead there means 
that this bead disappears from the abacus. 
It turns out that the reduction steps in this part of the algorithm can be described 
by a triple of ordinary partitions, again with the summed weight of these 
partitions equal to the number of bars removed \cite[Proposition 3.6]{kwon2006affine}. On the other hand, the final result can be encoded by a pair of ordinary partitions, or Young diagrams, which have the additional property of being a pyramid. 

The different bar removal steps (B1)-(B4) construct the map $c$ algorithmically and uniquely. The stated 
facts about parameterizing the steps prove the existence of the bijection~\eqref{Dpart_biject}. 
To complete the proof of~\eqref{Dcore_biject}, we only need to remark further that the set  
of ordinary Young diagrams having the shape of a pyramid is in bijection with the set of integers.
(see \cite[Remark 3.10(2)]{kwon2006affine}). This gives the remaining $\SZ^{2}$ factor in the bijection \eqref{Dcore_biject}.
\end{proof}

\begin{example}
A possible sequence of bar removals on the abacus and Young wall of Examples \ref{ex:ywall} and \ref{ex:abacusY} is as follows. Perform step (B1) on the beads at positions 7, 8, 10, 11. Perform step (B2) on the pairs of beads at positions (1,5) and (2,4). Perform step (B4) four times on two beads at position 12 by moving them consecutively to positions 9, 6 (where they take the color black), 3 and then 0 (which means removing them from the abacus). The resulting abacus configuration is then
\begin{center}
\begin{tikzpicture}

\draw (0,5) node {$R_1$};
\draw (1,5) node {$R_{2}$};
\draw (2,5) node {$R_{3}$};
\draw (3,5) node {$R_{4}$};
\draw (4,5) node {$R_{5}$};
\draw (5,5) node {$R_{6}$};

\draw (-0.5,4.8) -- (5.5,4.8);
\draw (2.5,5.3) -- (2.5,3);
\draw (1.5,5.3) -- (1.5,3);
\draw (4.5,5.3) -- (4.5,3);
\draw (5.5,5.3) -- (5.5,3);

\node at ( 0,4.5)  {1};
\node at ( 1,4.5)  {2};
\node at ( 2,4.5)  {3};
\node at ( 3,4.5)  {4};
\node at ( 4,4.5)  {5};
\node at ( 5,4.5) [circle,draw,inner sep=0pt,minimum size=12pt,fill=gray!50] {6};
\node at ( 0,4) {7};
\node at ( 1,4) {8};
\node at ( 2,4) {9};
\node at ( 3,4) {10};
\node at ( 4,4) {11};
\node at ( 5,4) [circle,draw,inner sep=0pt,minimum size=12pt] {12};
\node at ( 0.5,3.5) {\vdots};
\node at ( 3.5,3.5) {\vdots};
\end{tikzpicture}
\end{center}
This configuration describes the following core Young wall:
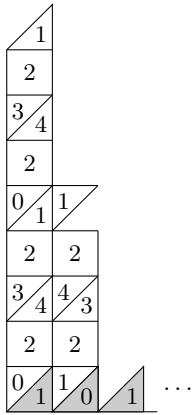
\begin{center}
	\begin{tikzpicture}[scale=0.6, font=\footnotesize, fill=black!20]
	
	\draw (0, 0) -- (0,8) -- (1,9)--(1,0);
	\draw (2,5)--(1,5);
	\draw (0, 8) -- (1,8);
	\draw (0, 7) -- (1,7);
	\draw (0, 6) -- (1,6);
	\draw (0, 5) -- (1,5);
	\draw (0, 4) -- (2,4);
	\draw (0, 3) -- (2,3);
	\draw (0, 2) -- (2,2);
	\draw (0, 1) -- (2,1);
	\draw (2, 0) -- (2,4);
	
	\draw (0,0) -- (3.3,0);
	\foreach \x in {0,1}
	{
		\draw (\x,2) -- (\x+1,3);
		\draw (\x,4) -- (\x+1,5);
		\draw (\x+0.5,1.5) node {2};
		\draw (\x+0.5,3.5) node {2};
		\filldraw (\x,0) -- (\x+1,1) -- (\x+1,0) -- cycle ;
	}
	\filldraw (2,0) -- (3,1) -- (3,0) -- cycle ;
	\filldraw (1,0) -- (2,1) -- (2,0) -- cycle ;
	\draw (2.75,0.35) node {1};
	\foreach \x in {0}
	{
		\draw (\x+0.25,0.65) node {0};
		\draw (\x+0.75,0.35) node {1};
		\draw (\x+0.25,2.65) node {3};
		\draw (\x+0.75,2.35) node {4};
		\draw (\x+0.25,4.65) node {0};
		\draw (\x+0.75,4.35) node {1};
	}
	\draw (1.25,0.65) node {1};
	\draw (1.75,0.35) node {0};
	\draw (1.25,2.65) node {4};
	\draw (1.75,2.35) node {3};
	\draw (1.25,4.65) node {1};
	
	\foreach \x in {0}
	{
		\draw (\x+0.25,6.65) node {3};
		\draw (\x+0.75,6.35) node {4};
		\draw (\x+0.75,8.35) node {1};
		\draw (\x,6) -- (\x+1,7);
	}

	\foreach \x in {0}
	{
		\draw (\x+0.5,5.5) node {2};
		\draw (\x+0.5,7.5) node {2};
	}
	\draw (3.8,0.5) node {\dots};
	\end{tikzpicture}
\end{center}
\end{example}

\section{Enumeration of Young walls}
\label{seq:enumyw}

We next determine the multi-weight of a Young wall $Y$ in terms 
of the bijections \eqref{Dpart_biject}-\eqref{Dcore_biject}. 
The quotient part is easy: the multi-weight of each bar is $(1,1,2,\ldots, 2, 1,1)$,  
so the $(n+1)$-tuple of partitions contributes a factor of
\[ \left(\prod_{m=1}^{\infty}(1-q^m)^{-1}\right)^{n+1}. \]
Turning to 
cores, under the bijection $\CC_\Delta\leftrightarrow \SZ^n$, the total weight of a core Young wall $Y\in \CC_\Delta$ corresponding 
to $(z_1, \ldots, z_n)\in\SZ^n$ is calculated in \cite[Remark 3.10]{kwon2006affine}: 
\begin{equation} 
\label{eq:dncoreweight}
|Y|=\frac{1}{2}\sum_{i=1}^{n-2} \left((2n-2)z_i^2-(2n-2i-2)z_i\right)+(n-1)\sum_{i=n-1}^n\left(2z_{i}^2+z_{i}\right).
\end{equation}
The next result gives a refinement of this formula for the multi-weight of $Y$.
\begin{theorem} 
\label{thm:orbiserdn}	
Let $q=q_0q_1q_2^2\dots q_{n-2}^2q_{n-1}q_n$, corresponding to a single bar. 
\begin{enumerate}
\item Composing the bijection~\eqref{Dcore_biject} with an appropriate $\SZ$-change of coordinates in the lattice~$\SZ^n$, 
the multi-weight of a core Young wall $Y\in\CC_\Delta$ corresponding to an element $\mathbf{m}=(m_1,\dots,m_n) \in \SZ^n$
can be expressed as
\[q_1^{m_1}\cdot\dots\cdot q_n^{m_n}(q^{1/2})^{\mathbf{m}^\top \cdot C \cdot \mathbf{m}},\]
where $C$ is the Cartan matrix of type $D_n$.

\item The multi-weight generating series 
\[ Z_{\Delta}(q_0,\dots,q_n) = \sum_{Y\in \CZ_\Delta} \mathbf{q}^{\mathbf{wt}(Y)}\]
of Young walls for $\Delta$ of type $D_n$ can be written as
\begin{equation*} Z_{\Delta}(q_0,\dots,q_n)=\frac{\displaystyle\sum_{ \mathbf{m}=(m_1,\dots,m_n) \in \SZ^n }^\infty q_1^{m_1}\cdot\dots\cdot q_n^{m_n}(q^{1/2})^{\mathbf{m}^\top \cdot C \cdot \mathbf{m}}}{\displaystyle\prod_{m=1}^{\infty}(1-q^m)^{n+1}}.
\end{equation*}

\item The following identity is satisfied between the coordinates $(m_1,\dots,m_n)$ and $(z_1, \ldots, z_n)$ on~$\SZ^n$:
\begin{equation*} 
\sum_{i=1}^n m_i = -\sum_{i=1}^{n-2}(n-1-i)z_i-(n-1)c(z_{n-1}+z_n)-(n-1)b.
\end{equation*}
Here $z_1+\dots+z_{n-2}=2a-b$ for integers $a \in \SZ$, $b \in \{0,1\}$, and $c=2b-1 \in \{-1,1\}$.
\end{enumerate}
\end{theorem}

Statement (2) clearly follows from (1) and the discussion preceding Theorem \ref{thm:orbiserdn}.  Statement (3) is used to achieve additional results in \cite{gyenge2015euler}.

Let us write $z_I=\sum_{i \in I}z_i$ for $I\subseteq\{1,\dots,n-2\}$. Each such number decomposes uniquely as $z_I=2a_I-b_I$, where $a_I \in \SZ$ and $b_I \in \{0,1\}$. Let us introduce also $c_I=2b_I-1 \in \{-1,1\}$. We will make use of the relations
\[a_I=\sum_{i\in I}a_i-\sum_{i_1\in I, i_2 \in I\setminus\{i_1\}}b_{i_1}b_{i_2}+\sum_{i_1\in I, i_2 \in I\setminus\{i_1\},i_3 \in I\setminus\{i_1,i_2\}}2b_{i_1}b_{i_2}b_{i_3}-\dots\;, \]
\[b_I=\sum_{i\in I}b_i-\sum_{i_1\in I, i_2 \in I\setminus\{i_1\}}2b_{i_1}b_{i_2}+\sum_{i_1\in I, i_2 \in I\setminus\{i_1\},i_3 \in I\setminus\{i_1,i_2\}}4b_{i_1}b_{i_2}b_{i_3}-\dots\;. \]
To simplify notations let us introduce
\[r_I:=a_I-\sum_{i\in I}a_i=-\sum_{i_1\in I, i_2 \in I\setminus\{i_1\}}b_{i_1}b_{i_2}+\sum_{i_1\in I, i_2 \in I\setminus\{i_1\},i_3 \in I\setminus\{i_1,i_2\}}2b_{i_1}b_{i_2}b_{i_3}-\dots\;.\]

Using these notations the colored refinement of the weight formula \eqref{eq:dncoreweight} is the following.
\begin{lemma}
\label{lem:dncorecolorweight}
Given a core Young wall $Y\in\CC_\Delta$ corresponding to $(z_i)\in \SZ^n$ in the bijection of~\eqref{Dcore_biject}, 
its content is given by the formula
\begin{gather*}
\mathbf{q}^{\mathbf{wt}(Y)}=q_1^{-\sum_{i=1}^{n-2}b_i}q_2^{-2a_1-\sum_{i=2}^{n-2}b_i}\dots q_{n-2}^{-\sum_{i=1}^{n-3}2a_i-b_{n-2}}(q_0q_1^{-1}q_{n-1}q_n)^{-\sum_{i=1}^{n-2}a_i} (q_0q_1^{-1})^{a_{1\dots n-2}}\\
\cdot q^{\frac{1}{2}\sum_{i=1}^{n-2}(z_i^2+b_i)+z_{n-1}^2+z_n^2 } \\\cdot (q^{b_{1\dots n-2}} (q_1^{-1}\dots q_{n-2}^{-1} q_{n-1}^{-1})^{c_{1\dots n-2}})^{z_{n-1}}(q^{b_{1\dots n-2}} (q_1^{-1}\dots q_{n-2}^{-1} q_{n}^{-1})^{c_{1\dots n-2}})^{z_{n}}.
\end{gather*}
\end{lemma}
When forgetting the coloring a straightforward check shows that Lemma \ref{lem:dncorecolorweight} gives back \eqref{eq:dncoreweight}. Notice also that $z_i^2+b_i=4a_i^2-4a_ib_i+2b_i$ is always an even number, so the exponents are always integers. 
\begin{proof}[Proof of Lemma \ref{lem:dncorecolorweight}]
Suppose that we restrict our attention to blocks of color $i$ by substituting $q_j=1$ for $j\neq i$. 
Clearly,
\[ wt_i(Y) \leq \sum_j wt_j(Y) =|Y| \]
where $|Y|$ is the total weight of $Y$. 
This inequality is true for each $0\leq i \leq n$, and $|Y|$ is a linear combination of the parameters $\{z_i\}_{1 \leq i \leq n}$, their squares and a constant.
It follows from the definition of the $\{z_i\}_{1\leq i \leq n}$ that each $wt_i$
is a convex, increasing function of them. These imply that, when considered over the reals, each
$wt_i$ are at most quadratically growing, convex analytic functions of $\{z_i\}_{1\leq i \leq n}$. As
a consequence, each $wt_i$ is again a linear combination of constants, the parameters $\{z_i\}_{1\leq i \leq n}$ and
their products. It is enough to check that the claimed formula is correct in two cases:
\begin{enumerate}
\item when any of the $z_i$'s is set to a given number and the others are fixed to 0; and
\item when all of the parameters are fixed to 0 except for an arbitrary pair $z_i$ and $z_j$, $i\neq j$. 
\end{enumerate}

First, consider that $z_i\neq 0$ for a fixed $i$, and $z_j=0$ in case $j\neq i$.
\begin{enumerate}
\item[(a)] When $1 \leq i \leq n-2$, then the colored weight of the corresponding core Young wall is 
\[(q_1 \dots q_i)^{-b_i}(q_{i+1}^2 \dots q_{n-2}^2q_{n-1}q_n)^{-a_i}q^{2a_i^2-2a_ib_i+b_i}\;.\] 
\item[(b)] When $i \in \{n-1, n\}$, then the associated core Young wall has colored weight 
\[q^{z_i^{2}}(q_1 q_2\dots q_{n-2} q_{i})^{z_i}\;.\]
\end{enumerate}
Both of these follow from (\ref{eq:dncoreweight}) and its proof in \cite{kwon2006affine} by taking into account the colors of the blocks in the pattern.

Second, assume that $z_i$ and $z_j$ are nonzero, but everything else is zero. Then the total weight is not the product of the two individual weights, but some correction term has to be introduced. The particular cases are:
\begin{enumerate}
\item[(a)] $1 \leq i,j \leq n-2$. There can only be a difference in the numbers of $q_0$'s and $q_1$'s which comes from the fact that in the first row there are only half blocks with 0's in the odd columns and 1's in the even columns. Exactly $-r_{ij}$ blocks change color from 0 to 1 when both $z_i$ and $z_j$ are nonzero compared to when one of them is zero. In general, this gives the correction term $(q_0q_1^{-1})^{r_{1\dots n-2}}=(q_0q_1^{-1})^{a_{1\dots n-2}-\sum_{i=1}^{n-2}a_i}$.
\item[(b)] $1 \leq i \leq n-2$, $j \in \{n-1,n\}$. For the same reason as in the previous case the parity of $z_i$ modifies the colored weight of the contribution of $z_j$, but not the total weight of it. If $z_i$ is even, then the linear term of the contribution of $z_j$ is $q_1 q_2\dots q_{n-2} q_{j}$. In the odd case it is $q_0 q_2\dots q_{n-2} q_{\kappa(j)}$. This is encoded in the correction term $(q^{b_{1\dots n-2}} (q_1^{-1}\dots q_{n-2}^{-1} q_{j}^{-1})^{c_{1\dots n-2}})^{z_j}$.
\item[(c)] $i=n-1$, $j=n$. $z_{n-1}$ and $z_{n}$ count into the total colored weight completely independently, so no correction term is needed.
\end{enumerate}

Putting everything together gives the claimed formula for the colored weight of an arbitrary core Young wall.

\end{proof}


Now we turn to the proof of Theorem \ref{thm:orbiserdn}. After recollecting the terms in the formula of Lemma \ref{lem:dncorecolorweight} it becomes
\begin{gather*}q_1^{-b_{1\dots n-2}-c_{1\dots n-2}(z_{n-1}+z_n)}\prod_{i=2}^{n-2}q_i^{-2a_{1\dots i-1}+c_{1\dots i-1}b_{i\dots n-2}-c_{1\dots n-2}(z_{n-1}+z_n)} \\
\cdot q_{n-1}^{-a_{1\dots n-2}-c_{1\dots n-2}z_{n-1}}q_{n}^{-a_{1\dots n-2}-c_{1\dots n-2}z_{n}} \\
\cdot q^{\sum_{i=1}^{n-2} (2a_i^2-2a_ib_i+b_i)+b_{1\dots n-2} z_{n-1}+z_{n-1}^2+b_{1\dots n-2} z_{n}+z_{n}^2+r_{1\dots n-2}}
\end{gather*}

Let us define the following series of integers:
\[
\begin{array}{r c l}
m_1 & = & -b_{1\dots n-2}-c_{1\dots n-2}(z_{n-1}+z_n)\;, \\
m_2 & = & -2a_1+c_1b_{2\dots n-2}-c_{1\dots n-2}(z_{n-1}+z_n)\;, \\
& \vdots & \\
m_{n-2} & = & -2a_{1 \dots n-3}+c_{1\dots n-3}b_{n-2}-c_{1\dots n-2}(z_{n-1}+z_n)\;, \\
m_{n-1} & = & -a_{1 \dots n-2}-c_{1\dots n-2}z_{n-1}\;, \\
m_{n} & = & -a_{1 \dots n-2}-c_{1\dots n-2}z_n \;.
\end{array}
\]
It is an easy and enlightening task to verify that the map
\[ \SZ^n \rightarrow \SZ^n, \quad (z_1,\dots,z_n) \mapsto (m_1,\dots,m_n)\]
is a bijection, which is left to the reader.

\begin{proof}[Proof of Theorem \ref{thm:orbiserdn}] (1): One has to check that
\begin{multline*} \sum_{i=1}^n m_i^2-m_1m_2-m_2m_3-\dots-m_{n-2}(m_{n-1}+m_n)=\\
=\sum_{i=1}^{n-2} (2a_i^2-2a_ib_i+b_i)+b_{1\dots n-2} z_{n-1}+z_{n-1}^2+b_{1\dots n-2} z_{n}+z_{n}^2+r_{1\dots n-2}\;.
\end{multline*}

The terms containing $z_{n-1}$ or $z_n$ on the left hand side are
\begin{gather*}  (n-2)(z_{n-1}+z_n)^2+z_{n-1}^2+z_n^2-(n-3)(z_{n-1}+z_n)^2-z_{n-1}^2-z_n^2-2 z_{n-1}z_n\\
+\left(2b_{1\dots n-2}+\sum_{i=1}^{n-3}2(2a_{1 \dots i}-c_{1\dots i}b_{i+1 \dots n-2})+2a_{1 \dots n-2}\right)c_{1\dots n-2}(z_{n-1}+z_n)\\
-\left(b_{1\dots n-2}+\sum_{i=1}^{n-3}2(2a_{1 \dots i}-c_{1\dots i}b_{i+1 \dots n-2})+2a_{1 \dots n-2}\right)c_{1\dots n-2}(z_{n-1}+z_n)\\
=b_{1\dots n-2} z_{n-1}+z_{n-1}^2+b_{1\dots n-2} z_{n}+z_{n}^2\;,
\end{gather*}
since $b_{1\dots n-2}c_{1\dots n-2}=b_{1\dots n-2}$.

The terms containing neither $z_{n-1}$ nor $z_n$ on the left hand side are
\begin{gather*}
b_{1\dots n-2}+\sum_{i=1}^{n-3}(2a_{1 \dots i}-c_{1\dots i}b_{i+1 \dots n-2})^2+2a_{1 \dots n-2}^2 -b_{1\dots n-2}(2a_1-c_1b_{2\dots n-2}) \\
-\sum_{i=1}^{n-4}(2a_{1 \dots i}-c_{1\dots i}b_{i+1 \dots n-2})(2a_{1 \dots i+1}-c_{1\dots i+1}b_{i+2 \dots n-2}) \\
-2(2a_{1 \dots n-3}-c_{1\dots n-3}b_{n-2})a_{1 \dots n-2}\;.
\end{gather*}
\begin{lemma}
\label{lem:abc}
\[ 2a_{1 \dots i}-c_{1\dots i}b_{i+1 \dots n-2}=\sum_{j=1}^{i}(2a_{j}-b_{j})+b_{1 \dots n-2}\;,\]
\end{lemma}
\begin{proof}
\begin{gather*} 2a_{1 \dots i}-c_{1\dots i}b_{i+1 \dots n-2} \\ =2a_{1\dots i-1}+2a_{i}-2b_{1\dots i-1}b_{i}+c_{1\dots i-1}c_{i}b_{i+1 \dots n-2} \\
=2a_{1\dots i-1}+2a_{i}-2b_{1\dots i-1}b_{i}+2c_{1\dots i-1}b_{i}b_{i+1 \dots n-2}-c_{1\dots i-1}b_{i+1 \dots n-2} \\
=2a_{1\dots i-1}+2a_{i}-b_{i}-c_{1\dots i-1}(b_{i+1 \dots n-2}+b_{i}-2b_{i}b_{i+1 \dots n-2})\\
= 2a_{1\dots i-1}-c_{1\dots i-1}b_{i \dots n-2}+2a_{i}-b_{i}\;,
\end{gather*}
and then use induction.
\end{proof}

Applying Lemma \ref{lem:abc} and the last intermediate expression in its proof to the terms considered above, they simplify to
\begin{gather*}
b_{1\dots n-2}+\sum_{i=1}^{n-3}(2a_{1 \dots i}-c_{1\dots i}b_{i+1 \dots n-2})^2+2a_{1 \dots n-2}^2 -b_{1\dots n-2}(2a_1-c_1b_{2\dots n-2}) \\
-\sum_{i=1}^{n-4}(2a_{1 \dots i}-c_{1\dots i}b_{i+1 \dots n-2})(2a_{1\dots i}-c_{1\dots i}b_{i+1 \dots n-2}+2a_{i+1}-b_{i+1})\\
-2(2a_{1 \dots n-3}-c_{1\dots n-3}b_{n-2})a_{1 \dots n-2} \\
=b_{1\dots n-2}+(2a_{1 \dots n-3}-c_{1\dots n-3}b_{n-2})^2+2a_{1 \dots n-2}^2 -b_{1\dots n-2}(2a_1-c_1b_{2\dots n-2}) \\
-\sum_{i=1}^{n-4}(2a_{1 \dots i}-c_{1\dots i}b_{i+1 \dots n-2})(2a_{i+1}-b_{i+1})-2(2a_{1 \dots n-3}-c_{1\dots n-3}b_{n-2})a_{1 \dots n-2} \\
=b_{1\dots n-2}+(2a_{1 \dots n-3}-c_{1\dots n-3}b_{n-2})(2a_{1 \dots n-3}-c_{1\dots n-3}b_{n-2}-2a_{1 \dots n-2})\\
+2a_{1 \dots n-2}^2 -b_{1\dots n-2}(2a_1-c_1b_{2\dots n-2})-\sum_{i=1}^{n-4}(2a_{1 \dots i}-c_{1\dots i}b_{i+1 \dots n-2})(2a_{i+1}-b_{i+1})\\
=b_{1\dots n-2}+2a_{1 \dots n-2}^2 -b_{1\dots n-2}(2a_1-c_1b_{2\dots n-2}) \\
-\sum_{i=1}^{n-3}(2a_{1 \dots i}-c_{1\dots i}b_{i+1 \dots n-2})(2a_{i+1}-b_{i+1}) \\
=2a_{1 \dots n-2}^2 -b_{1\dots n-2}(2a_1-b_1)-\sum_{i=1}^{n-3}\left(\sum_{j=1}^{i}(2a_{j}-b_{j})+b_{1 \dots n-2}\right)(2a_{i+1}-b_{i+1})\;.
\end{gather*}
Let us denote this expression temporarily as $s_{n-2}$. Taking into account that 
\[a_{1 \dots n-2}=a_{1\dots n-3}+a_{n-2}-b_{1 \dots n-3}b_{n-2}\;,\] 
\[b_{1 \dots n-2}=b_{1\dots n-3}+b_{n-2}-2b_{1 \dots n-3}b_{n-2}\;,\] 
$s_{n-2}$ can be rewritten as
\begin{gather*}
2a_{1\dots n-3}^2+2a_{n-2}^2+2b_{1\dots n-3}b_{n-2}+4a_{1\dots n-3}a_{n-2} \\ 
-4a_{1\dots n-3}b_{1\dots n-3}b_{n-2}-4a_{n-2}b_{1\dots n-3}b_{n-2}\\
-(b_{1\dots n-3}+b_{n-2}-2b_{1\dots n-3}b_{n-2})(2a_1-b_1) \\
-\sum_{i=1}^{n-4}\left(\sum_{j=1}^{i}(2a_{j}-b_{j})+b_{1 \dots n-3}\right)(2a_{i+1}-b_{i+1})\\
-\sum_{i=1}^{n-3}(b_{n-2}-2b_{1\dots n-3}b_{n-2})(2a_{i+1}-b_{i+1})\\
-\left(\sum_{j=1}^{n-3}(2a_{j}-b_{j})+b_{1 \dots n-3}\right)(2a_{n-2}-b_{n-2})\\
=s_{n-3}+2a_{n-2}^2+2b_{1\dots n-3}b_{n-2}+4a_{1\dots n-3}a_{n-2} \\
-4a_{1\dots n-3}b_{1\dots n-3}b_{n-2}-4a_{n-2}b_{1\dots n-3}b_{n-2}\\
-(b_{n-2}-2b_{1\dots n-3}b_{n-2})\left(\sum_{i=1}^{n-2}2a_i-b_i \right)-\left( \sum_{j=1}^{n-3}(2a_j-b_j)+b_{1\dots n-3}\right)(2a_{n-2}-b_{n-2})\\
=s_{n-3}+2a_{n-2}^2-2a_{n-2}b_{n-2}+b_{n-2}+2a_{n-2}\left(2a_{1\dots n-3}-\sum_{j=1}^{n-3}(2a_j-b_j)\right)\\
+b_{1\dots n-3}b_{n-2}-2a_{n-2}b_{1\dots n-3}\\
=s_{n-3}+2a_{n-2}^2-2a_{n-2}b_{n-2}+b_{n-2}-b_{1\dots n-3}b_{n-2}\;,
\end{gather*}
where at the last equality the identity $\sum_{j=1}^{n-3}(2a_j-b_j)=z_{1\dots n-3}=2a_{1\dots n-3}-b_{1\dots n-3}$ was used.

It can be checked that $s_1=2a_{1}^2-2a_{1}b_{1}+b_{1}$, so induction shows that
\[s_{n-2}=\sum_{i=1}^{n-2} (2a_i^2-2a_ib_i+b_i+b_{1\dots i-1}b_{i})\;.\]
It remains to show that 
\[\sum_{i=2}^{n-2}b_{1\dots i-1}b_{i}=r_{1\dots n-2}\;,\]
which requires another induction argument, and is left to the reader.

(3): Apply Lemma \ref{lem:abc} on
\[\sum_{i=1}^n m_i = -b_{1\dots n-2}-\sum_{i=1}^{n-2}(2a_{1 \dots i-1}-c_{1 \dots i-1}b_{i \dots n-2})-2a_{1\dots n-2}-(n-1)c_{1\dots n-2}(z_{n-1}+z_{n}).\]

\end{proof}

\bibliographystyle{amsplain}
\bibliography{dncomb}

\providecommand{\bysame}{\leavevmode\hbox to3em{\hrulefill}\thinspace}
\providecommand{\MR}{\relax\ifhmode\unskip\space\fi MR }
\providecommand{\MRhref}[2]{%
  \href{http://www.ams.org/mathscinet-getitem?mr=#1}{#2}
}
\providecommand{\href}[2]{#2}
\begin{thebibliography}{10}

\bibitem{bialynicki1973some}
A.~Bia{\l}ynicki-Birula, \emph{{Some theorems on actions of algebraic groups}},
  Ann. of Math. (1973), 480--497.

\bibitem{bryan2019g}
J.~Bryan and \'A. Gyenge, \emph{{$ G $-fixed Hilbert schemes on $ K3 $
  surfaces, modular forms, and eta products}}, arXiv preprint arXiv:1907.01535
  (2019).

\bibitem{cartan1957quotient}
H.~Cartan, \emph{{Quotient d’un espace analytique par un groupe
  d’automorphismes}}, Algebraic geometry and topology: a symposium in honor
  of S. Lefschetz, Princeton University Press, Princeton, NJ, 1957,
  pp.~90--102.

\bibitem{fujii2005combinatorial}
Sh. Fujii and S.~Minabe, \emph{{A Combinatorial Study on Quiver Varieties}},
  SIGMA. Symmetry, Integrability and Geometry: Methods and Applications
  \textbf{13} (2017), 052.

\bibitem{gottsche1990betti}
L.~G{\"o}ttsche, \emph{{The Betti numbers of the Hilbert scheme of points on a
  smooth projective surface}}, Math. Ann. \textbf{286} (1990), no.~1, 193--207.

\bibitem{gusein2004power}
S.~M. Gusein-Zade, I.~Luengo, and A.~Melle-Hern{\'a}ndez, \emph{A power
  structure over the {G}rothendieck ring of varieties}, Math. Res. Lett.
  \textbf{11} (2004), no.~1, 49--57. \MR{MR2046199 (2004m:14038)}

\bibitem{gyenge2016phdthesis}
{\'A}.~Gyenge, \emph{{Hilbert schemes of points on some classes of surface
  singularities}}, Ph.D. thesis, E\"otv\"os L\'or\'and University, Budapest,
  2016.

\bibitem{gyenge2017enumeration}
\bysame, \emph{Enumeration of diagonally colored young diagrams}, Monatshefte
  f{\"u}r Mathematik \textbf{183} (2017), no.~1, 143--157.

\bibitem{gyenge2015announcement}
{\'A}.~Gyenge, A.~N{\'e}methi, and B.~Szendr{\H{o}}i, \emph{{Euler
  characteristics of Hilbert schemes of points on surfaces with simple
  singularities}}, Int. Math. Res. Not. \textbf{2017} (2016), no.~13,
  4152--4159.

\bibitem{gyenge2015euler}
\bysame, \emph{{Euler characteristics of Hilbert schemes of points on simple
  surface singularities}}, Europ. J. Math. \textbf{4} (2018), no.~2, 439--524.

\bibitem{james1981representation}
G.~James and A.~Kerber, \emph{{The Representation Theory of the Symmetric
  Group}}, Encyclopedia Math. Appl. (1981).

\bibitem{kang2004crystal}
S.-J. Kang and J.-H. Kwon, \emph{{Crystal bases of the Fock space
  representations and string functions}}, J. Algebra \textbf{280} (2004),
  no.~1, 313--349.

\bibitem{kwon2006affine}
J.-H. Kwon, \emph{{Affine crystal graphs and two-colored partitions}}, Lett.
  Math. Phys. \textbf{75} (2006), no.~2, 171--186.

\bibitem{nagao2009quiver}
K.~Nagao, \emph{{Quiver varieties and Frenkel--Kac construction}}, J. Algebra
  \textbf{321} (2009), no.~12, 3764--3789.

\bibitem{nakajima2002geometric}
H.~Nakajima, \emph{{Geometric construction of representations of affine
  algebras}}, Proceedings of the International Congress of Mathematicians
  (Beijing, 2002), vol.~1, IMU, Higher Ed, 2002, pp.~423--438.

\end{thebibliography}

\end{document}